\documentclass[12pt]{amsart}
 \theoremstyle{plain}
\newtheorem{theorem}{Theorem}[section]

\theoremstyle{definition}
\newtheorem{definition}[theorem]{Definition}

\newtheorem{example}[theorem]{Example}
\usepackage{amssymb,amscd,amsmath,graphicx,floatflt}
\usepackage[all]{xy}

\begin{document}

\title{On long knots in the full torus}

\author{Sera Kim}

\address{Sera Kim, Republic of Korea Naval Academy (Republic of Korea)\\
srkim85@gmail.com}

\author{Seongjeong Kim}

\address{Seongjeong Kim, Jilin University (China) \\
Bauman Moscow State Technical University (Russia) \\
ksj19891120@gmail.com}

\author{Vassily O. Manturov}

\address{Vassily O. Manturov, Moscow institute of physics and technology (Russia) \\ North-Eastern University (China)\\
Kazan Federal University (Russia)\\
vomanturov@yandex.ru}

\maketitle

\begin{abstract} The aim of this paper is to realise the techniques of picture-valued invariants and invariants valued in free groups for long knots in the full torus. Such knots and links are of a particular interest because of their relation to Legendrian knots, knotoids, $3$-manifolds and many other objects. Invariants constructed in the paper are powerful and easy to compare. This paper is a sequel of \cite{ma}. Long knots naturally appear in the study of classical knots \cite{fi,mo}.

\vspace{0.40in} {\em Keywords:\/} knots on cylinder; full torus; braids; Brunnian braids; picture; free-group; picture-valued invariant.\\
{\sl AMS(MOS) Subj. Class:\/} Primary 57M25. Secondary 57M27, 57K31.
\end{abstract}

\section{Introduction}
Virtual knot theory \cite{ka} is a proper generalisation of classical knot theory. Topologically, virtual knots are described as knots in thickened surfaces $S_g \times [0,1]$ modulo stabilisations and destabilisations \cite{ma2012}. Combinatorially, they can be described by admitting a new type of crossing, called {\em virtual crossing}, modulo generalised Reidemeister moves.

Virtual knots enjoy lots of properties never seen in classical knot theory. One of such properties is that if a virtual knot diagram $D$ is {\em complicated enough then it realises itself.}  \cite{ma2010} In some cases this means that if a diagram is {\em locally minimal} (can not be decreased in one step) then it is {\em globally minimal} (any diagram $D'$ of it \textquotedblleft contains $D$ inside\textquotedblright). This makes virtual knots similar to {\em words in free groups.}

This goal is achieved in \cite{ma2010} for {\em free knots}\footnote{We'll not need a formal definition of a free knot or a virtual knot in this paper.} (a drastic simplification of virtual knots) by constructing invariants valued in {\em pictures} or {\em linear combinations of knot diagrams}. In particular, in \cite{ma2010}, the third named author constructed an invariant $[\cdot]$ of free knots (it can be easily extended to virtual knots in various ways where the bracket of a diagram gives a linear combinations of diagrams obtained from it by some \textquotedblleft simplifications\textquotedblright.) For some diagrams we have a formula $$[D]=D;$$ here $D$ in the LHS is a particular knot diagram and $D$ in the RHS is just $D$ with coefficient $1$.

For $D'\equiv D$ we have $[D']=D$ which means that $D$ can be obtained from $D'$ by some {\em simplifications}.

In \cite{ma}, a (long) free knot invariant is constructed. This invariant is valued in a free product of some copies \footnote{In most of cases, such invariants can be extended to those valued in free groups by adding some (co)orientation; but the main effect is visible at the level of free products of $\mathbb{Z}_{2}$.} of $\mathbb{Z}_{2}$.

The reason why the bracket works so nicely is the existence of {\em parity} for classical knots. The parity is a way of distinguishing between {\em even} and {\em odd crossings} in such a way that some axioms are satisfied when crossings undergo Reidemeister moves. We are going to apply virtual techniques to classical knots and links.

The main problem was the absence of parity for classical knots and links \cite{fi} which is caused by the fact that $H_{1}(\mathbb{R}^{2},\mathbb{Z}_{2}) = \{0\}$. In this paper, we deal with knots in the full torus (thickened cylinder) where we use the parity coming from $H_{1}(S^{1}, \mathbb{Z}_{2}) = \mathbb{Z}_{2}$.

In contrast to the knot group (where we map {\em knot} $\mapsto$ {\em group}), we construct one specific group\footnote{Actually, once a good example is constructed, there are lots of ways to upgrade such invariants. In \cite{ma}, a bare count of crossings is generalised to some elements in a nice group.} which is pleasant to work with and invariants of our knots are valued in this concrete group.

The paper is organised as follows. In Section 2, we construct long knot diagrams on the cylinder. In Section 3 introduce {\em linear Gauss diagrams} for long knot diagrams on the cylinder and the {\em words} by assigning the letters to the endpoints of chords in the linear Gauss diagram. These words in a group $G''$ can be {\em invariants} for long knot diagrams on the cylinder. Section 4 is devoted to examples. We show that the invertibility property can be easily captured by one glance at the word in the group $\bar{G}$.\\

{\bf Acknowledgments.} The first author was supported by Basic Science Research Program through the National Research Foundation of Korea (NRF) funded by the Ministry of Education (NRF-2021R1I1A3045371). The second and the third named authors were supported by the Russian Foundation for Basic Research (grants No. 20-51-53022, 19-51-51004).

The authors would like to express their heartfelt gratitude to L. H. Kauffman, S. G. Gukov and I. M. Nikonov for their interest and fruitful discussions on the present work.

\section{Motivation}

A knot is a (smooth) embedding of a circle $S^1$ in $3$-dimensional space $\mathbb{R}^3$ up to isotopy and a knot diagram is a projection on $\mathbb{R}^2$ with the under/over information. Two knot diagrams are equivalent if there is a finite sequence of Reidemeister moves from one diagram to the other. A link is an embedding of several circles in $\mathbb{R}^3$.

A {\em long knot diagram} $K$ is an embedding from $\mathbb{R}$ to $\mathbb{R}^2$ such that there is a real number $r$ so that $f(x) = (x,0)$ for any real number x and $|x|>r$. We also consider a {\em long knot diagram} $K$ with the real number $r$ on a double-punctured sphere as Fig. \ref{fig:ex5_0}. Actually, it can be considered as a long knot diagram on the cylinder.

\begin{figure}[ht]
\begin{center}
\resizebox{0.80\textwidth}{!}{%
\includegraphics{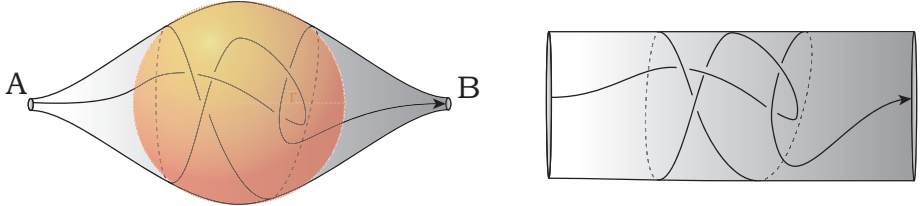} }
\caption{Long knot diagram on the cylinder} \label{fig:ex5_0}
\end{center}
\end{figure}

For a $2$-component link $L = K_1 \cup K_2$ in $S^3$, if one of the two components is unknotted, say $K_1$, the $K_2$ lies in the solid torus $T$ which is knot complement $S^3 \setminus N(K_1)$ of $K_1$. The Hopf link diagram is shown in Fig. \ref{fig:ex1} which is for the Hopf link $L_2a_1$.

\begin{figure}[ht]
\begin{center}
\resizebox{0.8\textwidth}{!}{%
\includegraphics{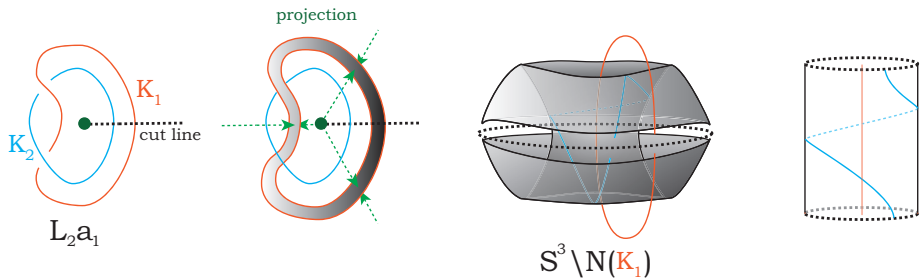} }
\caption{How to construct the long knot diagram on the cylinder} \label{fig:ex1}
\end{center}
\end{figure}

Then the projection of $K_2$ to $\partial T$ is the long knot diagram on the cylinder by cutting along the longitude of $T$. If $K_2$ is equivalent to a trivial component of $L$ as $K_1$, we could investigate the two long knot diagrams from a 2-component link $L$.

The group-valued invariants in the present paper can be extended to {\em closed knots} in the full torus (without free ends); the corresponding invariants will be conjugacy classes of elements in the groups similar to $G''$ and $\bar{G}$ etc.

A crucial fact is that in this way one can construct not only {\em isotopy invariants} but also {\em cobordism invariants}.

We will devote a separate paper to the cobordisms of knots in the full torus.

We add yet one more motivation for the study of knots in the full torus and their cobordisms. We'll write a sequel to this paper. One of the crucial tools is the study of pairs of ``similar'' knots one of which is slice and the other is not. This search is performed by using surgeries along links (hence, one can use a link component as a source of homology). Hence sliceness of knots and links in the full torus is of a great importance for modern topology and new handy sliceness obstructions are very actual.

\section{The group presentations for long knots in the full torus}
In this paper, our knots are in the full torus whence diagrams are on the cylinder.
For a long knot diagram $K$ on the cylinder, we consider the {\em compact knot diagram} $\bar{K}$ from $K$ which is a knot diagram by connecting the infinity points of $K$. The {\em Gauss diagram} for $\bar{K}$ is a circle together with oriented $m$ chords connecting $2m$ points on the circle where the points are related to the crossing information as Fig. \ref{fig:chord}.

\begin{figure}[ht]
\begin{center}
\resizebox{0.70\textwidth}{!}{%
\includegraphics{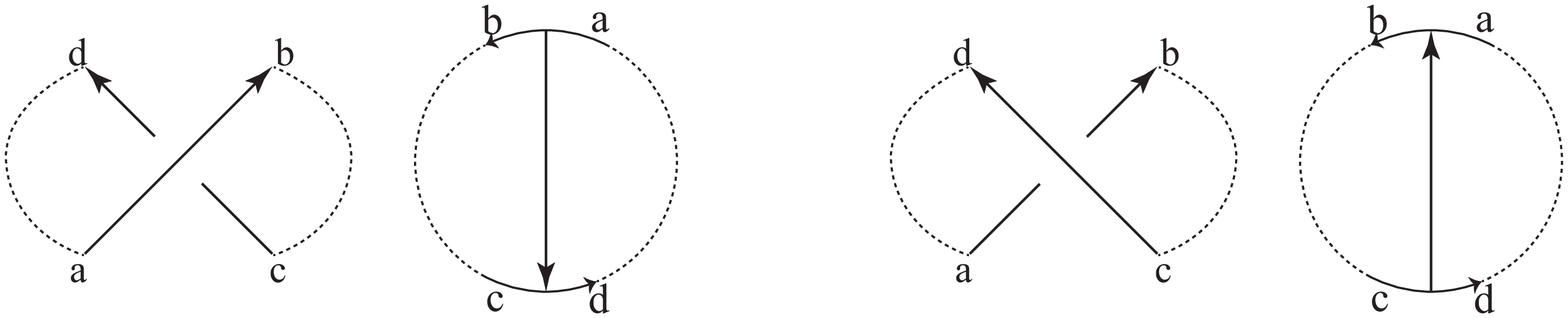} }
\caption{A crossing of a knot diagram and the corresponding chord of its Gauss diagram} \label{fig:chord}
\end{center}
\end{figure}

Similarly, by considering the Gauss diagram for the compact knot diagram $\bar{K}$ with the infinity points for $K$, we introduce the {\em linear Gauss diagram} for $K$ as shown in Fig. \ref{fig:ex5_1}, and it is denoted by $G(K)$.

\begin{figure}[ht]
\begin{center}
\resizebox{0.60\textwidth}{!}{%
\includegraphics{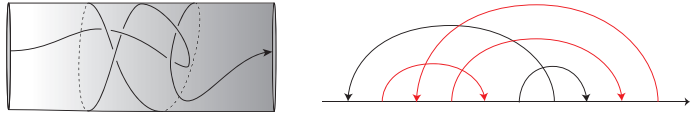} }
\caption{Long knot diagram and its linear Gauss diagram} \label{fig:ex5_1}
\end{center}
\end{figure}

In \cite{po}, Polyak introduced the generating set for Reidemeister moves. These moves can be translated to the moves for the Gauss diagrams as Fig. \ref{fig:ex7}. Polyak's moves are the same for linear diagrams of classical/virtual/whatever else knots.

\begin{figure}[ht]
\begin{center}
\resizebox{0.9\textwidth}{!}{%
\includegraphics{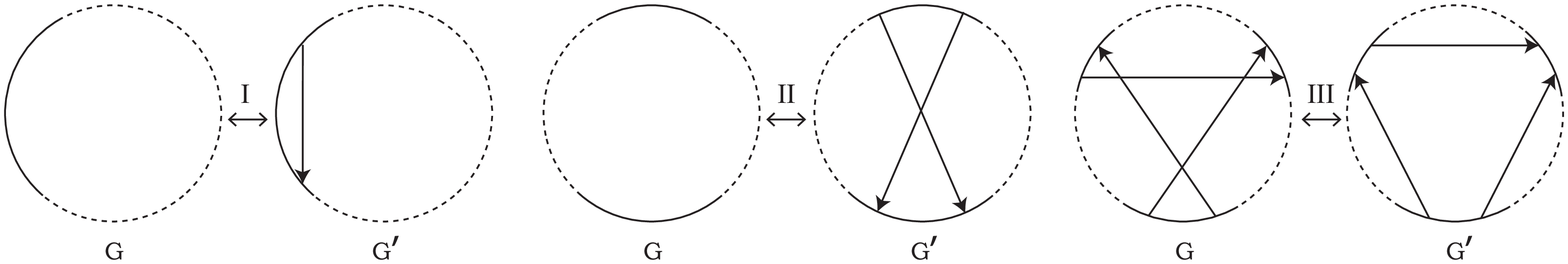} }
\caption{The generating set for Reidemeister moves} \label{fig:ex7}
\end{center}
\end{figure}

In the linear Gauss diagram $G(K)$ for a long knot diagram $K$ on the cylinder, we say that two chords of a linear Gauss diagrams are {\em linked} if their ends alternate. There are two subsets of chords via the {\em Gaussian parity}; one is the set of {\em odd chords} in $G(K)$ which are linked with oddly many chords in $G(K)$ and the other is the set of {\em even chords} in $G(K)$ which are linked with evenly many chords in $G(K)$.

Following \cite{ma}, we give the following definitions \ref{def1} of {\em types} and {\em positions} for the linear Gauss diagrams.

\begin{definition}\label{def1}\cite{ma}
Let $K$ be an oriented long knot diagram on the cylinder and $G(K)$ be a linear Gauss diagram for $K$. We enumerate the endpoints as they appear according to the orientation; we say that an odd chord is {\em of the first type} if it is linked with evenly many even chords and {\em of the second type} if it is linked with oddly many even chords.

Moreover, we enumerate the endpoints of chords in $G(K)$ along the orientation of $K$. The endpoints are called to be {\em in odd position} if the number of endpoints containing itself from the starting point of $G(K)$ is odd. Otherwise the endpoint is called to be {\em in even position}.
\end{definition}

From Definition \ref{def1}, we denote each endpoint of a chord depending on parity and their position by $a,b^{\pm 1}, c^{\pm 1}$ as follows:
\begin{itemize}
\item The endpoint is denoted by $a$ if it is related to an even crossing.

\item The endpoint in even position is denoted by $b$ if it is related to the over odd crossing which is of the first type.

\item The endpoint in even position is denoted by $b'$ if it is related to the over odd crossing which is of the second type.

\item The endpoint in even position is denoted by $c$ if it is related to the under odd crossing which is of the first type.

\item The endpoint in even position is denoted by $c'$ if it is related to the under odd crossing which is of the second type.

\item For endpoints in odd position, we use the same letters with negative exponents:

$b^{-1}$, $(b')^{-1}$, $c^{-1}$, $(c')^{-1}$.

\end{itemize}

\begin{example}
Let $K$ be an oriented long knot from the link $L_{7}a_{1}$ and let $G(K)$ be the linear Gauss diagram of $K$, see Fig.~\ref{fig:ex4}. The red chords of $G(K)$ are even chords, and the black chords are odd chords. Since each odd chord intersects with even chords (red chords) oddly many times, all odd chords are of the second type. Now let us associate letters $a,b^{\pm 1},b'^{\pm 1},c^{\pm 1},c'^{\pm 1}$. With the first (leftmost) chord end, being even position undercrossing, we associate $c'$. With the next chord end (odd position overcrossings) we associate $(b')^{-1}$. The next chord end gives rise to $(c')^{-1}$. The final chord end gives rise to $b'$. See Fig. \ref{fig:ex4}.

\begin{figure}[ht]
\begin{center}
\resizebox{0.90\textwidth}{!}{%
\includegraphics{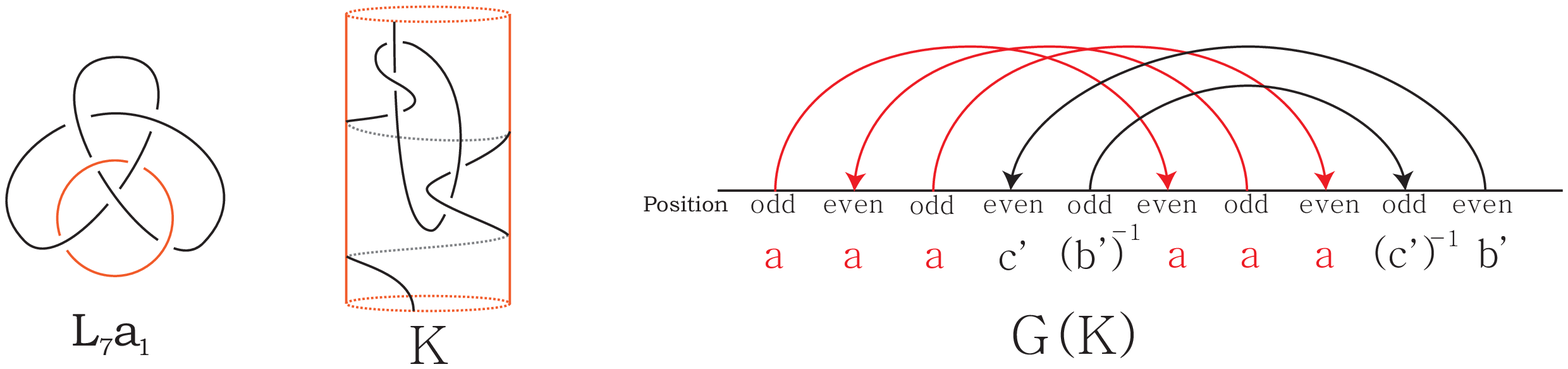} }
\caption{A long knot diagram $K$ on the cylinder obtained from the link $L_{7}a_{1}$ and the linear Gauss diagram $G(K)$ from $K$} \label{fig:ex4}
\end{center}
\end{figure}
\end{example}

Let us define the group $G''$ by
$$G'' = \langle a,b,b',c,c' | a^2=1, ab=(b')^{-1}a, ac=(c')^{-1}a, cb^{-1}=b'(c')^{-1}, b^{-1}c=(c')^{-1}b'\rangle$$
$$=\langle a,b,b',c,c' | a^2 =1, b'=ab^{-1}a, c'=ac^{-1}a, c'=bc^{-1}b', c'=b'c^{-1}b \rangle$$
$$=\langle a,b,c | a^2 =1, bac=cab \rangle=\langle a,d,e | a^2 =1, de=ed \rangle$$ where $d=bc^{-1}$ and $e=ba$. It is equivalent to $\mathbb{Z}_{2} * (\mathbb{Z} \oplus \mathbb{Z})$.

Its Cayley graph is shown in Fig. \ref{fig:cayley}.

\begin{figure}[ht]
\begin{center}
\resizebox{0.50\textwidth}{!}{%
\includegraphics{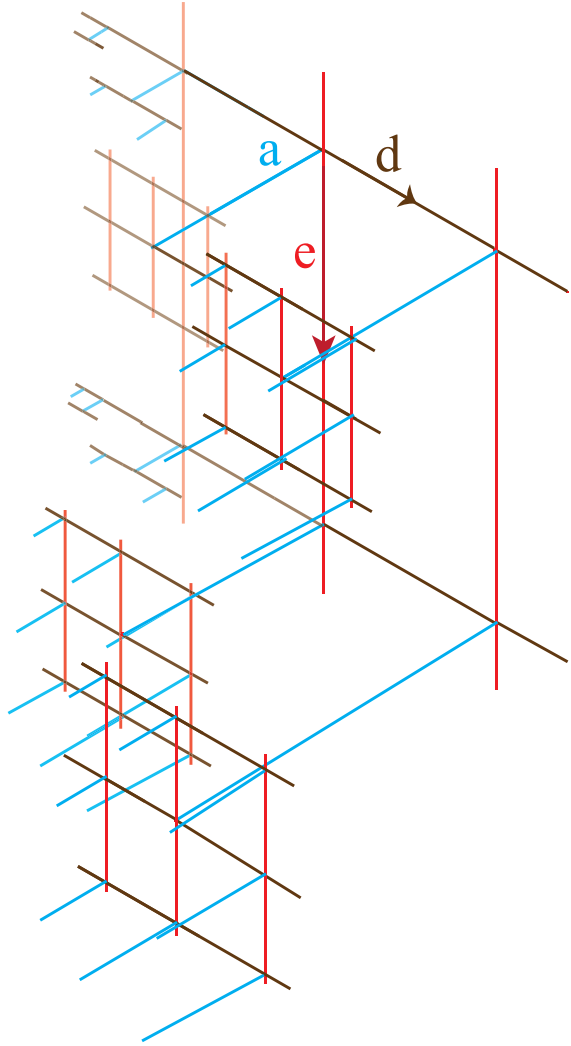} }
\caption{The Cayley graph for $G''$} \label{fig:cayley}
\end{center}
\end{figure}

Now we construct a map $\phi''$ from the set of oriented long knot diagrams on the cylinder to the group $G''$: for each oriented long knot diagram $K$ on the cylinder, the word $\phi''(K)$ in $G''$ is obtained by using letters $a,b(=ea),b'(=e^{-1}a),c(=d^{-1}ea),c'(=e^{-1}da)$ along the orientation of $K$, which correspond to ends of chords as described in above. For example, for the oriented long knot $K$ in the full torus, see  Fig.~\ref{fig:ex4} chord end gives rise to one of $a,b^{\pm 1},c^{\pm 1},(b')^{\pm 1},(c')^{\pm 1}$ as described in Fig.~\ref{fig:ex4}. This yields
$$\phi''(K) = aaac'(b')^{-1}a(c')^{-1}b'\equiv ae^{-1}ded^{-1}a \equiv 1 \in G''.$$

\begin{theorem}\label{thm1}
Let $\phi''$ be the map defined as above. Then $\phi''(K)$ is an invariant of oriented long knots $K$ in the full torus.
\end{theorem}

\begin{proof}
For an oriented long knot diagram $K$ on the cylinder, let $K'$ be an oriented long knot diagram from $K$ by a single Reidemeister move so that the number of crossings of $K'$ is more than or equal to the number of crossings of $K$.

For the $R_1$ move, the new chord of $G(K')$ is an even chord and its consecutive endpoints are denoted by $a$. Then $\phi''(K')$ is same as the word by adding the word $a^2$ into $\phi''(K)$ and $a^2$ is trivial in the group presentation $G''$. That is, $\phi''(K) \equiv \phi''(K')$ in $G''$.

For the $R_2$ move, we have two new chords of $G(K')$. The respective letters cancel in the group $G''$.

For the $R_3$ move, we denote the corresponding chords in $G(K)$ by $d_1$, $d_2$ and $d_3$, and the corresponding chords in $G(K')$ are denoted by $e_1$, $e_2$ and $e_3$. Then all parities for positions of endpoints of $d_{i}$ differ from those for $e_{i}$ for $i=1,2,3$. If $d_{i}$ is an odd chord of the first type (the second type), then $e_{i}$ is an odd chord and of the second type (the first type), respectively. Then we get the four relations $ab=(b')^{-1}a$, $ac=(c')^{-1}a$, $cb^{-1}=b'(c')^{-1}$, and $b^{-1}c=(c')^{-1}b'$.
\end{proof}

\section{The group presentation for long knots in the full torus II}

In the previous section, we associate letters and their inverses with endpoints of chords in the linear Gauss diagram for long knot diagrams on the cylinder. Specifically, the exponent ($1$ or $-1$) depends on the position of the endpoint.

From now on, we introduce another way to define exponents of letters depending on the sign of chords (crossings). Let $K$ be an oriented long knot diagram on the cylinder and $G(K)$ be the linear Gauss diagram for $K$. For a chord $d$ in $G(K)$, the sign of $d$ is defined as the sign of the corresponding crossing in $K$. The sign of a chord $d$ is denoted by $i(d)$.

\begin{figure}[ht]
\begin{center}
\resizebox{0.35\textwidth}{!}{%
\includegraphics{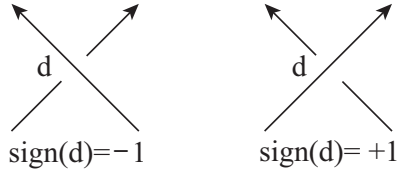} }
\caption{Signs for crossings} \label{fig:ex6_1}
\end{center}
\end{figure}

Similarly, let $K$ be an oriented long knot diagram on the cylinder and $G(K)$ be a linear Gauss diagram. For each chord $d$ in $G(K)$, we associate letters to endpoints of $d$ as follows:

\begin{itemize}
\item The endpoint of $d$ is denoted by $a^{i(d)}$ if it is related to the over even crossing.

\item The endpoint of $d$ is denoted by $b^{i(d)}$ if it is related to the over odd crossing and $d$ is of the first type.

\item The endpoint of $d$ is denoted by $(b')^{i(d)}$ if it is related to the over odd crossing and $d$ is of the second type.

\item The endpoint of $d$ is denoted by $a^{-i(d)}$ if it is related to the under even crossing.

\item The endpoints of $d$ is denoted by $c^{-i(d)}$ if it is related to the under odd crossing and $d$ is of the first type.

\item The endpoint of $d$ is denoted by $(c')^{-i(d)}$ if it is related to the under odd crossing and $d$ is of the second type.
\end{itemize}

For example, we consider the oriented long knot diagrams on the cylinder from the $2$-component link $L_{6}a_{1}$ in \cite{ka1}. This link gives rise to the word given in Fig. \ref{fig:ex6_0}.

\begin{figure}[ht]
\begin{center}
\resizebox{0.70\textwidth}{!}{%
\includegraphics{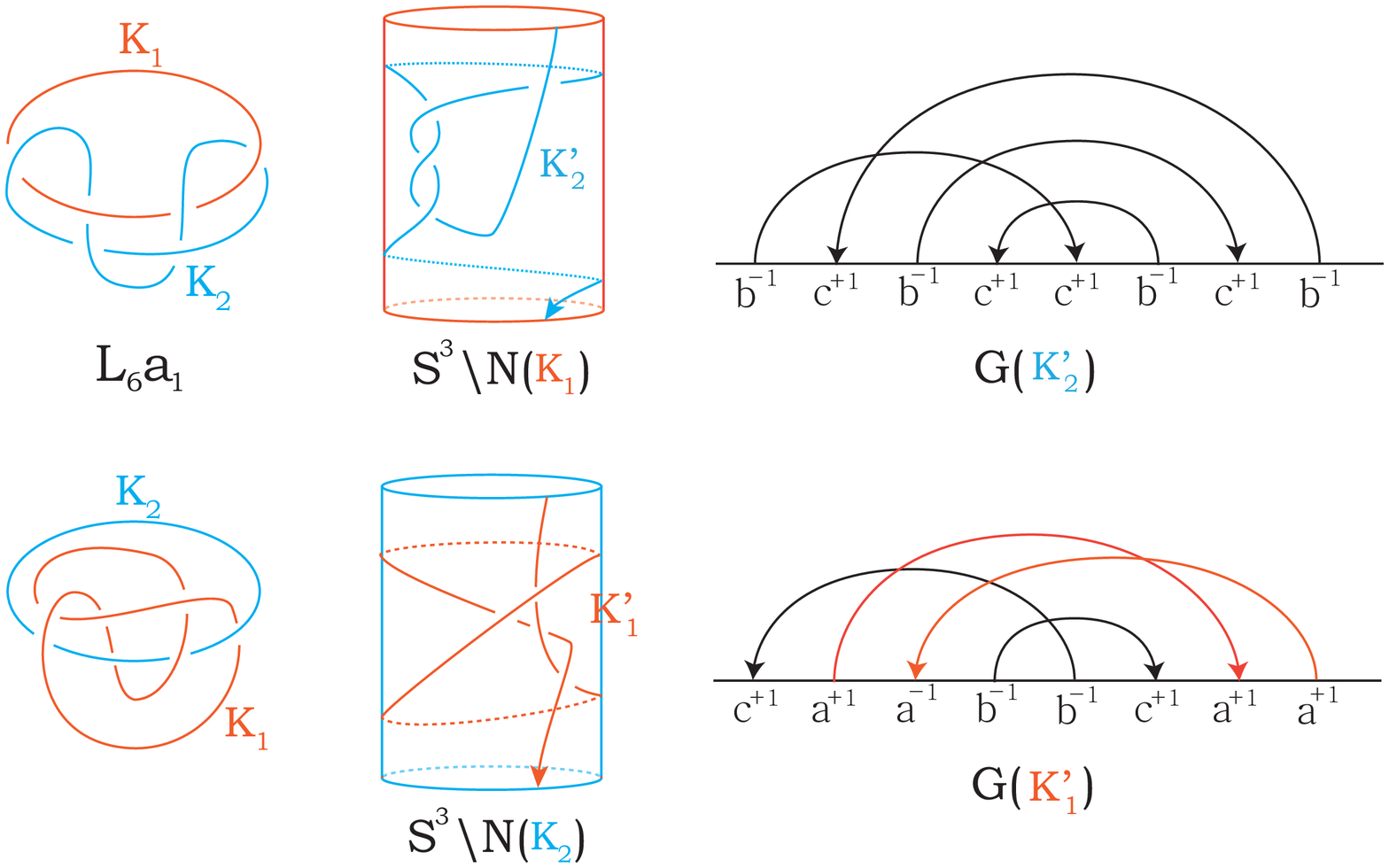} }
\caption{Example} \label{fig:ex6_0}
\end{center}
\end{figure}

Similarly to the previous section, we define the map ${\bar \phi}$ from long knot diagrams to the group ${\bar G}$.

$$\bar{G} = \langle a,b,b',c,c' |  b'=aba^{-1}=a^{-1}ba, c'=aca^{-1}=a^{-1}ca,$$
$$bc^{-1}=(c')^{-1}b', b'c^{-1}=(c')^{-1}b, b(c')^{-1}=c^{-1}b', b'(c')^{-1}=c^{-1}b \rangle =$$
$$=\langle a,b,c | a^{2}b=ba^{2}, a^{2}c=ca^{2}, b^{2}a=ab^{2}, c^{2}a=ac^{2}, abc=cba, cab=bac\rangle.$$

For the oriented long knot diagrams in the Fig. \ref{fig:ex6_0}, $$\bar{\phi}(K'_{2}) = (b^{-1}c)^{2}(cb^{-1})^{2}~~~~\mbox{and}~~~~\bar{\phi}(K'_{1}) = cb^{-1}b^{-1}c,$$ and they are both non-trivial since the sum of powers of $b$ and $c$ respectively for each word are nonzero.

\

Let $\bar{\phi}$ be defined as above; they are valued in the group $\bar{G}$. For an oriented long knot diagram $K$, $\bar{\phi}(K)$ is defined as a word consisting of letters on the linear Gauss diagram $G(K)$ along the orientation.

\begin{theorem}\label{thm2} $\bar{\phi}$ is an invariant of long knots in the full torus.
\end{theorem}

\begin{proof}
Let $G$ be a linear Gauss diagram for a long knot diagram on the cylinder, and let $G'$ be the linear Gauss diagram obtained from $G$ by a $R_1$-move in Fig. \ref{fig:ex7}. Then $\bar{\phi}(G')$ is equal to $\bar{\phi}(G)$ since the new chord of $G'$ makes the word as $aa^{-1}$.

Let $G$ be a linear Gauss diagram for a long knot diagram on the cylinder, and let $G'$ be the linear Gauss diagram obtained from $G$ by a $R_2$-move in Fig. \ref{fig:ex7}. Then $\bar{\phi}(G')$ is equal to $\bar{\phi}(G)$ since the new chords of $G'$ give rise to canceling letters.

Let $G$ be a linear Gauss diagram for a long knot diagram on the cylinder, and let $G'$ be the linear Gauss diagram obtained from $G$ by a $R_3$-move in Fig. \ref{fig:ex7}. The three chords of $G$ related to $R_3$ move have two cases; the first one is that they are all even chords, and the second case is that one of them is an even chord and others are two odd chords. In the first case, the letters at the endpoints are all $a$. Thus, $\bar{\phi}(G)$ is equal to $\bar{\phi}(G')$. But, in the second case, we get the relations
$a^{\pm 1}(c^{\bullet})^{\mp1}=(c^{\bullet})'^{\mp1}a^{\pm 1}, 
(c^{\bullet})^{\pm 1}a^{\mp1}=a^{\mp1}(c^{\bullet})'^{\pm1},
a^{\pm1}(b^{\bullet})^{\mp1}=(b^{\bullet})'^{\mp1}a^{\pm1}, 
(b^{\bullet})^{\pm1}a^{\mp1}= a^{\mp1}(b^{\bullet})'^{\pm1},
b^{\pm1}c^{\mp1}= (c')^{\mp1}(b')^{\pm1},
(b')^{\pm1}c^{\mp1}= (c')^{\mp1}b^{\pm1}, 
b^{\pm1}(c')^{\mp1}= c^{\mp1}(b')^{\pm1}, 
(b')^{\pm1}(c')^{\mp1}= (c')^{\mp1}(b')^{\pm1}$, where $b^{\bullet}$ ($c^{\bullet}$, resp.) is either $b$ ($c$, resp.) or $b'$ ($c'$, resp.), $(b')'=b$, and $(c')'=c$. Actually, from these relations, we get by the relations $b'=aba^{-1}=a^{-1}ba$, $c'=aca^{-1}=a^{-1}ca$, $a^{2}b=ba^{2}$, $a^{2}c=ca^{2}$,  $b^{2}a=ab^{2}$, $c^{2}a=ac^{2}$, $abc=cba$, and $cab=bac$.
\end{proof}

\section{Post scriptum}

Free groups (or free products of cyclic groups) are very pleasant from lots of points of view: word problem, conjugacy problem, invertibility etc. Solutions for many problems can be obtained by looking at just one element.

We present a certain way of encoding ``knotted objects'' by such easy groups (actually, with some loss of information; rather, we construct an invariant). We expect to tackle various problems in knot theory and low-dimensional topology (say, construct invariants of 3-manifolds) by looking at some elements in free (product of cyclic) groups.

\end{document}